\documentclass[12pt]{article}

\usepackage{algpseudocode}
\usepackage{amsmath} 
\usepackage{amssymb}
\usepackage{amsthm}

\newcommand{\Sym}{{\rm Sym}} 
\newcommand{\Alt}{{\rm Alt}}

\def\ndiv{ {\not\kern-.5pt\hbox{\big |}\,} }

\newtheorem{theorem}{Theorem}[section]
\newtheorem{lemma}[theorem]{Lemma}
\newtheorem{proposition}[theorem]{Proposition}

\newtheorem{definition}[theorem]{Definition}

\title{Finding blocks of imprimitivity\\
when there is a small-base action on blocks}

\author{Robert M. Beals\\
  IDA--CCR Princeton\\
  805 Bunn Drive\\
  Princeton, NJ 08440\\
  {\tt beals@idaccr.org}}

\begin{document}

\maketitle

\begin{abstract}
  Given a transitive
  permutation group $G$ of degree $n$, we seek to determine
  whether or not $G$ is primitive, and to find a system of blocks
  of imprimitivity in the case that $G$ is imprimitive.  An algorithm of
  Atkinson solves this problem in time $O(n^2)$, while a previous algorithm
  of ours runs in time $O(n\log^3|G|)$, which is advantageous
  in the small-base case.
  A simpler algorithm of Sch\"onert and Seress has the same
  asymptotic $O(n\log^3|G|)$ performance.

  In this paper we extend the small-base algorithms to work with
  imprimitive groups $G$ which, while not small-base in the
  action on $n$ points, possess a small-base action on a block system.
  Using a recent upper bound by Kelsey and
  Roney-Dougal on the size of a
  nonredundant base of a primitive group of a given degree,
  we obtain a time
  of $O(n\log^5 n)$ except in the case that
  $G$ has a primitive action
  (either on the $n$ points or on a block system) for which the socle
  is isomorphic to $\Alt(m)^d$ for some $m\geq 5$ and $d\geq 1$.
  
  A key component of our improvement is a new variant of
  sifting, which is a workhorse of permutation group
  algorithms. 

\end{abstract}

\section{Introduction}
\subsection{Background}
Let $\Omega$ be a finite set with $n$ elements, and let
$G\leq\Sym(\Omega)$ be given by list $S$ of generators.
In time $O(n|S|)$ (linear in the input length)
it is possible to determine whether $G$
is {\em transitive}\/, i.e., whether for all $\alpha,\beta\in\Omega$
there is an element $g\in G$ with $\alpha^g=\beta$.
In the transitive case, it is often of interest whether
$G$ preserves some nontrivial partition of $\Omega$ other than
the partition into singletons.  If so, $G$ is 
{\em imprimitive}\/, and the cells of such a partition
are {\em blocks of imprimitivity}\/. Otherwise $G$ is {\em primitive}\/.

An algorithm of Atkinson~\cite{atkinson1975algorithm} tests
primitivity and finds a block system in the imprimitive case.
This algorithm runs in time $O(n^2|S|)$, and is the fastest known
algorithm for arbitrary groups $G$.
There is much interest,
however, in {\em small-base}\/ groups, and for these it
is possible to find blocks more efficiently.

Let $B=(\beta_1,\ldots,\beta_\ell)$ be a list of elements of $\Omega$.
We say that $B$ is a {\em base}\/
of a permutation group $G\subseteq\Sym(\Omega)$
if only the identity of $G$
fixes $B$ pointwise.  A base $B$ is {\em nonredundant}\/
if it satisfies an additional {\em nonredundancy}\/ condition:
for all $i=1,\ldots,\ell$ there is an element $g_i\in G$
which fixes all $\beta_j$ with $j<i$ but which does not
fix $\beta_i$ (note that the issue of nonredundancy may
depend on the ordering of the list $B$).

We will call a prefix of a nonredundant
base a {\em nonredundant partial base}\/.  Note that a
list $B$ which satisfies the nonredundancy condition is either
a base or a proper prefix of a list which similarly satisfies the
nonredundancy condition: if such a $B$ is not a base, then
there is some $g\in G$ fixing $B$ pointwise yet moving some $\beta\in\Omega$;
appending this $\beta$ to $B$ preserves the nonredundancy condition.
Inductively, we see that if $B$ satisfies the nonredundancy condition,
it is a nonredundant partial base.

Within the time constraints
of our algorithm, we cannot compute a nonredundant base
or verify that a given 
list $B$ is a nonredundant base.  We will, however,
maintain a list $B=(\beta_1,\ldots,\beta_\ell)$,
together with group elements $g_1,\ldots,g_\ell$,
which witness the nonredundancy condition.  That is, for each $i$, the group
element $g_i$ fixes all $\beta_j$ with $j<i$, but does
not fix $\beta_i$.
In this situation we say that $B$ is 
a {\em certified}\/ nonredundant partial base,
and we shall call the list of $g_i$ the {\em certificate}\/.
The certificate demonstrates that $B$ is a
nonredundant partial base, and that $|B|$ is a lower bound
on the cardinality of a nonredundant base.

The concept of a small-base group  was formalized
by Babai, Cooperman, Finkelstein and Seress in~\cite{babai1991nearly},
motivated by computational interest in families of groups for which
the size of a nonredundant base is bounded by a polynomial in $\log n$.
Since any nonredundant base $B$ satisfies
$$\log|G|/\log n \leq|B|\leq\log|G|$$
(all logarithms in this paper are base two),
we may equivalently say that $\log|G|$ is bounded by a polynomial in $\log n$.
For such {\em small-base} groups, an
agorithm with a running time of $O(n\log^c|G|)$ (for some constant $c$)
is nearly linear.  The first such algorithm
for testing primitivity and
finding a block system in the imprimitive case was given by
Beals~\cite{Beals1991}, taking time $O(n\log^3|G|+n|S|\log|G|)$.
Subsequently Sch\"onert and Seress~\cite{SchonertSeress1994}
gave a much simpler algorithm with the same timing.

We emphasize that that the small-base primitivity algorithms
\cite{Beals1991,SchonertSeress1994}
are proven, worst-case results.
They do not rely on heuristics, or randomness,
or any assumptions regarding the generators.
Our new results are in this vein: we present
a deterministic algorithm with a worst-case
running time which is nearly linear for a larger
class of groups than the previous algorithms.

Suppose that we run one of the small-base primitivity algorithms
on a group which is not a small-base group.
The algorithm {\em may}\/ succeed in quickly determining whether
the group is primitive, but it may not.
In this event,
if we terminate the algorithm once some reasonable time has elapsed,
no useful information is obtained besides whatever lower
bound on $|G|$ can be inferred from mere fact of the algorithm not finishing.

In this paper we give a modification which causes the small-base
algorithm to produce more interesting output, in the form
of a certified nonredundant partial base, in the event that
it does not finish.  And we give a method to find blocks of imprimitivity
given a certified nonredundant partial base of sufficient size (i.e., exceeding
the maximum size of a nonredundant base of the action of $G$ on
some block system).

\subsection{Transversals and sifting}
Let $G\leq\Sym(\Omega)$.  Let
$G_\alpha$ denote the stabilizer in $G$ of some $\alpha\in\Omega$.
Both small-base primitivity algorithms~\cite{Beals1991,SchonertSeress1994}
start by invoking a subroutine of Babai et al~\cite{babai1991nearly}
which  constructs a transversal
$R$ for $G:G_\alpha$.
This transversal is stored as words in a smaller set of
elements $X$, with word length $\leq 2|X|\leq 2\log|G|$.
The time required is $O(n|X|^2+n|S|)=O(n\log^2|G|+n|S|)$.
The primitivity algorithm invokes this subroutine for $G$ itself, and for
up to $O(\log|G|)$ subgroups of $G$, for a total time of
$O(n\log^3|G|+n|S|)$ (the term involving $|S|$ only
affects the first running of the subroutine, and so is not multiplied by
$\log|G|$).
All of the improvements in this paper
stem from a re-working of this subroutine.

In our version we combine the computation of $R$ with a new variant
of Sims's sifting procedure~\cite{sims1970computational,sims1971computation,FurstHopcroftLuks1980,babai1991nearly,cooperman1990random,cooperman1992fast}
which we call {\em deep sifting}.  Let $\beta_1=\alpha$;
in the course of computing $R$, we may compute additional
base points $\beta_2,\beta_3,\ldots,\beta_\ell$.  
Following~\cite[Section 4.1]{seress-book},
we let $G^{[i]}$ denote the subgroup of $G$ which fixes
each of the first $i-1$ of these base points, so $G^{[1]}=G$
and $G^{[2]}=G_\alpha$.  For each $i=1,\ldots,\ell$,
we will have found a non-empty list
$X_i$ consisting of some elements of $G^{[i]}\setminus G^{[i+1]}$.
These lists will be $\leq \log n$ in length,
indeed with the property that
$$2^{|X_i|}\leq |G^{[i]}:G^{[i+1]}|,$$
so
$\sum_i|X_i|\leq\log|G|$.  The elements of
$R$ will be words of length at most $2\sum_i|X_i|\leq 2\log|G|$
in the elements of $X_1\cup X_2\cup\cdots X_\ell$.

Since deep sifting keeps each $|X_i|\leq\log n$, we cannot have
the word length grow too large without $\ell$ increasing.
That is, without
delving deeper into
the subgroup chain
$G^{[1]}\geq G^{[2]}\geq G^{[3]}\geq \ldots\geq 1$.

Both small-base primitivity
algorithms require, for at most $\log |G|$ subgroups $K$ of $G$,
transversals for $K:K_\alpha$.  With the above version of sifting,
the subgroups $K$ which arise will always contain $X_2,X_3,\ldots,X_\ell$,
so these lists can be re-used, yielding a mild improvement in the running time:
\begin{theorem}\label{mild}
  Let $G\leq\Sym(\Omega)$ be a
  transitive group given by
  a set $S$ of generators. In
  time $O(n\log^2|G|\log n+n|S|\log |G|)$,
  we can determine whether or not $G$ is
  primitive, and in the imprimitive case
  we can find a block system.
\end{theorem}
The main advantage, however, of the deep sifting approach is that
we can now handle groups other than small-base groups.
The key is that if we attempt, using deep sifting, to compute a transversal
for $G:G_\alpha$ for some group $G$ which is not
a small-base group, but terminate if the number  $\ell$ of base points
exceeds some threshold $L$, then one of two things will occur
in time $O(nL^2\log^2 n +n|S|)$:
\begin{enumerate}
\item The algorithm may compute the transversal.

\item The algorithm may produce a certified nonredundant partial
  base (for the certificate take each
  $g_i$ arbitrarily from the corresponding $X_i$)
  of size $L+1$.
\end{enumerate}
Further, if we plug this procedure into the small-base primitivity
algorithm, along with a limit on the base size, we achieve
something similar: either the primitivity
algorithm runs to completion in time $O(nL^2\log^3 n+n|S|L\log n)$
or we obtain a certified nonredundant partial base of size $L+1$.
We show in Section~\ref{actions} that from the certificate
we can find a block system in time $O(n|S|L\log n)$, 
as long as $L$ is an upper bound on the size of a nonredundant
base for the action of $G$ on {\em some}\/ block system.  If there is
more than one block system on which $G$ acts, we may take
$L$ to be the minimum over all block systems of the maximum
nonredundant base size.  This minimum is 
at most the minimum over 
block systems on which $G$ acts primitively.
So we are very interested in upper bounds on nonredundant
base sizes for primitive groups.

\subsection{Base sizes of primitive groups}
There is a rich history of bounds on base sizes of
various types of primitive
groups~\cite{bochert1892ueber,wielandt1969permutation,%
PraegerSaxl1980,Babai1981Uni,babai1982order,MR679977,Pyber1993361}.
Cameron~\cite{Cameron1981} laid the groundwork for applying results from the
Classification of Finite Simple Groups to this problem
via the O'Nan--Scott Theorem~\cite{scott1980representations,Liebeck_Praeger_Saxl_1988}.

There is one class of
primitive groups that do not have small bases:
A primitive group $G\leq\Sym(n)$ is {\em large}\/ if there are
positive integers $k$, $d$, and $m$ such that the socle of $G$ is
permutation-isomorphic to $\Alt(m)^d$ acting on ordered $d$-tuples of
$k$-element subsets of an $m$-set.  The large groups were identified
by Liebeck~\cite{liebeck1984minimal}\cite[Theorem 5.6C]{MR1409812}
 as the only primitive groups whose
smallest base size exceeds $9\log n$.
Liebeck's logarithmic bound has been
improved~\cite{HALASI201916,RoneyDougal2020}, most recently 
by Moscatiello and
Roney-Dougal~\cite{MoscatielloRoneyDougal2022} to
$\lceil\log n\rceil+1$,
with the sole exception being the Matthieu group $M_{24}$ acting on 24 points
with a base size of 7.

Naively, one might think that
a primitive group which is not large might
have its largest nonredundant base size be on the order of $\log^2 n$.
Recently, Kelsey and
Roney-Dougal~\cite{Kelsey2022-xi},
proving a conjecture of Gill, Lod\`a and Spiga~\cite{GILL_LODA_SPIGA_2022},
showed that this cannot occur:
the size of any nonredundant base for such a group is bounded by $5\log n$.
Exploiting this,
we set $L=5\log n$ to obtain our main result:

\begin{theorem}\label{main}
  Let $G\leq\Sym(\Omega)$ be a
  transitive group given by
  a set $S$ of generators.  In time $O(n\log^5 n+n|S|\log^2 n)$
  we can do one of the following:
  \begin{enumerate}
  \item Establish that $G$ is primitive.
  \item Establish that $G$ is imprimitive, finding a
    block system for $G$.
  \item Establish that all primitive actions of $G$
    on nontrivial partitions of $\Omega$ are large.
  \end{enumerate}
  (In the third case we do not identify a partition on which
  $G$ acts primitively, nor do we determine whether $G$ is primitive.)
\end{theorem}

To handle some large primitive actions, we can set $L=(9/2)n^{1/3}$
and obtain a sub-quadratic time of
$$O(n^{5/3}\log^3 n+n^{4/3}|S|\log n).$$
This succeeds for all but a few cases;
we give details in Theorem~\ref{five thirds} in Section~\ref{conclusion}.

\subsection{Organization}
This paper is organized as follows.  In Section~\ref{deep}
we describe our variant of sifting,
and describe the computation
of the transversals used by the primitivity algorithm.
In Section~\ref{primitivity} we describe how the small-base
primitivity algorithm is affected when modified to use deep sifting, and prove
Theorem~\ref{mild}. For simplicity, we focus on the
algorithm of Sch\"onert and Seress~\cite{SchonertSeress1994},
following the treatment of~\cite{seress-book}. In Section~\ref{actions}
we consider the case that $G$ is imprimitive, with a small-base
action on a block system, and complete the proof of Theorem~\ref{main}.
We also give details on the large primitive actions
for which we achieve sub-quadratic running time.

\section{Deep Sifting}\label{deep}
\subsection{Cubes}
Like Babai et al~\cite{babai1991nearly}, we use an algorithmic version
of the {\em cube}\/
construction of Babai and Szemer\'edi~\cite{BabaiSzemeredi1984}
to compute our transversals.
\begin{definition}[\cite{BabaiSzemeredi1984}]
  For a list $X=(x_1,x_2,\ldots,x_j)$ of elements of $G$,
  we denote by $C(X)$ set
  $$\{{x_1}^{\epsilon_1}{x_2}^{\epsilon_2}\cdots{x_j}^{\epsilon_j}\mid \epsilon_1,\epsilon_2,\ldots \epsilon_j\in\{0,1\}\},$$
  which we call the {\em cube}\/ generated by $X$.
\end{definition}
Elements of $C(X)$ may be represented by bit strings of length $|X|$,
or by words of length $\leq |X|$; we ignore these details, and simply say
that the algorithm encodes elements of $C(X)$ as words.
The encoded permutation, if needed, can be found with at most $|X|$
group operations in time $O(n|X|)$.
Often, however, we will only need to compute the
images of a few points of $\Omega$, and this is
accomplished much more quickly.
\begin{proposition}\label{image}
For any $\beta\in\Omega$
and $g\in C(X)$, we can compute $\beta^g$ in time $O(|X|)$
from the word encoding $g$.
\end{proposition}
We are also interested in how
$C(X)$ maps subsets of $\Omega$.
\begin{proposition}\label{subset}
For any set
$\Delta\subseteq\Omega$,
we can find the set $\Delta^{C(X)}$ in time $n|X|$.
If desired, for all $\gamma\in \Delta^{C(X)}$, we can in
that same time find an $\beta\in \Delta$ and
a word representing an element $g$ of $C(X)$
with $\beta^g=\gamma$.
\end{proposition}
Denote by  $X_1,X_2$ the concatenation of the
lists $X_1$ and $X_2$. The product $C(X_1)C(X_2)$
is thus the cube $C(X_1,X_2)$.  If we let $X^{-1}$ denote the reverse
of the list of inverses of the list $X$, then we have
$C(X)^{-1}=C(X^{-1})$.  So the set $C(X)^{-1}C(X)$ is the
cube $C(X^{-1},X)$

We call a cube $C(x_1,\ldots,x_j)$ {\em non-degenerate}\/ if it
contains $2^j$ distinct group elements.  Clearly the cube
$C(x_1)$ is non-degenerate iff $x_1$ is not the identity.
\begin{proposition}[\cite{BabaiSzemeredi1984}]
  For $j>1$,
  the cube $C(x_1,\ldots,x_j)$ is non-degenerate iff both
  of the following hold:
  \begin{itemize}
    \item The
      cube $C(x_1,\ldots,x_{j-1})$ is non-degenerate.
    \item The element $x_j$ satisfies
      $x_j\not\in C(x_1,\ldots,x_{j-1})^{-1}C(x_1,\ldots,x_{j-1})$.
  \end{itemize}
\end{proposition}
Determining whether $x_j\not\in C(x_1,\ldots,x_{j-1})^{-1}C(x_1,\ldots,x_{j-1})$
seems daunting.
In~\cite{babai1991nearly}, an efficiently testable sufficient condition is
given.  Recall that
by Proposition~\ref{subset} we can, for any $\gamma\in\Omega$,
compute the set $\gamma^{C(X)^{-1}C(X)}$
in time $O(n|X|)$.
If for some $g\in G$,
we have that $\gamma^g\not\in\gamma^{C(X)^{-1}C(X)}$, then certainly 
$g\not\in C(X)^{-1}C(X)$.

\subsection{Deep sifting}
We first describe the
data structure for deep sifting.
At any point in time, we have some base points
$\beta_1,\ldots,\beta_\ell$, and for each $i$
we have a set $X_i\subseteq G^{[i]}$ such that the
cube $C(X_i)$ is non-degenerate,
with its $2^{|X_i|}$ distinct group elements
mapping $\beta_i$ to distinct points.
Let $\Delta_i=\beta_{i}^{C(X_i)}$;
for each $\lambda\in\Delta_i$ we have a word
representing the element $r$ of $C(X_i)$ with ${\beta_i}^r=\lambda$.

Given an element $g$, we wish to either factor $g$
as a product of cube elements
in a particular manner (\ref{deepsift-fact}), or
augment one of the $X_i$ so that, after sifting,
$g$ can be so factored:
\begin{equation}\label{deepsift-fact}
  g \in C(X_1)^{-1}C(X_2)^{-1}\cdots C(X_\ell)^{-1}C(X_\ell)C(X_{\ell-1})\cdots C(X_1)
\end{equation}
We now describe the deep sifting process itself.
Given a group element $g$:
while $g\not\in G^{[\ell+1]}$,
let $i$ be the least such that ${\beta_i}^g\neq\beta_i$.
If ${\Delta_i}^g$ is disjoint from $\Delta_i$, we append $g$ to
$X_i$, update $\Delta_i$, and
terminate.
Otherwise we select some $\lambda\in{\Delta_i}^g \cap \Delta_i$,
and construct $s,t\in C(X_i)$ with $\beta_{i}^s=\lambda^{g^{-1}}$
and $\beta_{i}^t=\lambda$.  We replace $g$ by $s g t^{-1}\in G^{[i+1]}$
and proceed to the next iteration.
See the pseudocode in Figure~\ref{deepsift}.

\begin{figure}
\begin{algorithmic}
\Procedure{DeepSift}{$g$}
  \While{$g\not\in G^{[\ell+1]}$}
    \State $i\gets\min\{i \mid {\beta_i}^g\neq\beta_i\}$
    \If{${\Delta_i}^g \cap \Delta_i= \emptyset$}
      \State  Append$(X_i,g)$
      \State $\Delta_i\gets\Delta_i \cup {\Delta_i}^g$
      \State {\bf return}
    \Else
      \State  Select $\lambda\in{\Delta_i}^g \cap \Delta_i$
      \State Construct $s\in C(X_i)$ with ${\beta_i}^s=\lambda^{g^{-1}}$
      \State Construct $t\in C(X_i)$ with ${\beta_i}^t=\lambda$
      \State $g\gets s g t^{-1}$
    \EndIf
  \EndWhile
  \If {$g\neq 1$}
    \State Select $\beta_{\ell+1}$ in the support of $g$
    \State $X_{\ell+1}\gets (g)$
    \State $\Delta_{\ell+1}\gets\{ \beta_{\ell+1}, \beta_{\ell+1}^g\}$
    \State $\ell\gets\ell+1$
  \EndIf
  \State {\bf return}
\EndProcedure
\end{algorithmic}
\caption{Pseudocode for deep sifting}\label{deepsift}
\end{figure}
Note that since we only append to the list $X_i$ if doing so keeps
$|\Delta_i|=2^{|X_i|}$, we have $|X_i|\leq\log|G^{[i]}:G^{[i+1]}|$.  This means
that the $O(n|X_i|)$ cost of constructing $s$ and $t$ for a given $i$ value is
$O(n\log|G^{[i]}:G^{[i+1]}|)$, and
the total time required for sifting an element is thus $O(n\log|G|)$.

\subsection{Features of Deep Sifting}
We will henceforth use ``sifting'' to refer to deep sifting,
although we will occasionally use the term ``deep sifting'' as well.

Let us denote by ${X_i}^*$ the concatenation of
the lists $X_\ell,X_{\ell-1},\ldots,X_i$.
Note that sifting an element $g$ involves at most
$2|{X_1}^*|$ group operations, and is accomplished in time $O(n|{X_1}^*|)$.
Also, the group generated by $g$ and ${X_1}^*$ does not change
during the sifting process.  Sifting an element either
has no effect on the $X_i$ or appends a single element to
a single $X_i$, and so increases $|{X_1}^*|$ by one.

Let us call $C(X_i)$ the {\em shallow}\/ cube at level $i$,
and $C({X_i}^*)^{-1}C({X_i}^*)$
the {\em deep}\/ cube at level $i$.  Only the shallow cube is
used during the sifting process.  But a typical purpose of
cubes
is to obtain a transversal for $G^{[i]}:G^{[i+1]}$, and if deep sifting
is used then we will see that 
transversal elements can be found 
in the deep cube. Note that the
product of cubes in (\ref{deepsift-fact}) is the
deep cube at level 1.

Let $\Omega_i$ denote the set of images of $\beta_{i}$
by elements of the deep cube at level $i$ (so $\Omega_i$ is to the
deep cube what $\Delta_i$ is to the shallow cube). Any particular
$\Omega_i$ can be computed in time $O(n(|{X_i}^*|)$.
By Proposition~\ref{subset}, 
this can be accomplished while storing, for all $\lambda\in\Omega_i$,
a word of length $\leq 2|{X_i}^*|$
representing an element of the deep cube which maps $\beta_i$ to $\lambda$.
In our application we only require these words for $i=1$;
for $\lambda\in\Omega_1$ we denote by $r_\lambda$ the element of the
deep cube $C({X_1}^*)^{-1}C({X_1}^*)$ mapping $\beta_1$ to $\lambda$.
\begin{lemma}\label{augment-transversal}
  Suppose we sift an element $g\in G^{[i]}$ with $\beta_{i}^g=\lambda$.
  Then after sifting completes,
  we will have $\lambda\in\Omega_i$.
\end{lemma}
This follows immediately from:
\begin{lemma}\label{augment-generators}
  Suppose we sift an element $g\in G^{[i]}$. Then after sifting completes,
  we will have that $g$ is an element of the deep cube
  at level $i$.
\end{lemma}
\begin{proof}
  Clearly this is true if
  ${\Delta_i}^{g}\cap\Delta_i=\emptyset$, so assume otherwise.
  Also, assume $g_i\not\in G^{[i+1]}$; since
  the deep cube at level $i$ contains the deep cube at
  level $i+1$, this assumption is without loss of generality.
  Let $s$ and $t$ be the elements of $C(X_i)$ used,
  so that for the next iteration we replace $g$ by
  $g'=sgt^{-1}$.  The effect on ${X_1}^*$ of
  sifting $g$ is identical to the effect of
  sifting $g'$, and 
  inductively we have that
  $g'$ will belong to the deep cube at level $i+1$ after
  sifting is complete.  Expressing $g$ in terms of $g'$
  we see that
  \begin{eqnarray*}
    g&=&s^{-1}g't\\
    &\in& s^{-1} C(X_{i+1}^*)^{-1}C(X_{i+1}^*) t\\
    &\in& C(X_i)^{-1} C(X_{i+1}^*)^{-1}C(X_{i+1}^*) C(X_i)\\
    &=& C({X_i}^*)^{-1}C({X_i}^*)\\
  \end{eqnarray*}
  as desired.
\end{proof}
We emphasize that
deep sifting an element $g\in G^{[i]}$ {\em may}\/
increase $|{X_i}^*|$ even if $g$ was already in the deep cube 
$C({X_i}^*)^{-1}C({X_i}^*)$.  

As noted above, each application of the deep sift either
increases $|{X_1}^*|$ by one or has no effect.  
If we sift an element $g$
with $\beta_1^g\not\in\Omega_1$,
then by Lemma~\ref{augment-transversal}, $|{X_1}^*|$ increases.
Similarly, sifting an element 
$g\in G^{[2]}\setminus \langle {X_2}^*\rangle$ 
causes $|{X_2}^*|$, and thus $|{X_1}^*|$,
to increase, by Lemma~\ref{augment-generators}.
At all times we have $|{X_1}^*|\leq\log|G|$ and,
since each $|X_i|\leq\log n$,
we also have $|{X_1}^*|\leq\ell\log n$.

\begin{lemma}\label{transversal}
  Let $L$ be an arbitrary positive integer,
  let $G\leq\Sym(\Omega)$ be a
  group given by
  a set $S$ of generators, and let $\alpha$ be a given point not
  fixed by $G$.
  Let $\mu=\min(L\log n,\log|G|)$.
  In time $O(n\mu^2+n|S|)$ we can
  either compute a transversal for $G:G_\alpha$
  with words of length $\leq 2\mu$, or compute
  a partial nonredundant base  $\beta_1,\ldots,\beta_{L+1}$ of size $L+1$,
  together with, for each $i=1,\ldots,L+1$, an
  element $g_i\in G^{[i]}\setminus G^{[i+1]}$.
\end{lemma}
\begin{proof}
  We describe the algorithm.
  Select some $g\in S$
  with $\alpha^g\neq\alpha$.
  We initialize the
  deep sifting data structure with
  $\beta_1=\alpha$, set $X_1$ to be
  the single element list $(g)$, and set $\ell=1$.
  
  As long as $\ell\leq L$ and
  $\Omega_1\neq \alpha^G$, we find some element
  $\lambda\in\Omega_1$ and some generator $g\in S$
  with $\lambda^g\not\in\Omega_1$,
  and sift $r_\lambda g$.
  Established methods~\cite{seress-book}
  limit the total time used in
  selecting $\lambda\in\Omega_1$ and $g\in S$
  to $O(n|S|)$.

  The tasks of computing $\Omega_1$,
  computing $r_\lambda g$,
  and sifting $r_\lambda g$ each take time $O(n|{X_1}^*|)$;
  by the remarks above, this is $O(n\mu)$.
  By Lemma~\ref{augment-transversal},
  each such sift results in an increase to the
  deep cube, so the number of such sifts is $O(\mu)$.
  The total sifting time is thus $O(n\mu^2)$.
  
  At termination, we either have a transversal,
  or we have $\ell=L+1$. In the latter case
  we can then take an arbitrary $g_i$ from each $X_i$.
\end{proof}
For $L\log n>\log |G|$, a transversal is produced,
in the time given by~\cite[Lemma 4.4.2]{seress-book}.
Note that with deep sifting, when
${\beta_1}^g\not\in\Omega_1$, we do not
simply append $g$ to $X_1$. This would double
$|C(X_1)|$ but not necessarily $|\Delta_1|$. Rather, we increase
$|{X_1}^*|$ by 1 in such a way that one of the $|\Delta_i|$ doubles.
In this way, we maintain a certified nonredundant partial base $B$,
satisfying $|B|\geq|{X_1}^*|/\log n$ at all times.

\section{Primitivity testing with deep sifting}\label{primitivity}
We will run the primitivity algorithm of
Sch\"onert and Seress~\cite{SchonertSeress1994,seress-book},
modified so that the computation of
transversals uses deep sifting, and terminating
the computation if the base size exceeds some threshold $L$.
Following the notation of~\cite[Section 5.5]{seress-book},
in this Section we will for the most part
use the notation $\alpha$ instead of $\beta_1$
and $G_\alpha$ instead of  $G^{[2]}$, until we need to specifically
refer to the remaining $\beta_i$ and  $G^{[i]}$ in Lemma~\ref{truncated} and
in Section~\ref{actions}.
As before, let $\mu$ denote $\min(L\log n,\log|G|)$;
as long as the number $\ell$ of base points found is
at most $L$, we have $|{X_1}^*|\leq\mu$.

So that the reader not be expected to have~\cite{seress-book} in hand,
we summarize those details of~\cite[Section 5.5]{seress-book} which are
relevant to the current discussion.  
The only parts of the algorithm which
cause the running time to depend on $|G|$ are:
\begin{itemize}
  \item[(a)] Computing the transversal $R$, taking time $O(n\log^2|G|+n|S|)$.
  \item[(b)] Finding, for $\leq n\log n$ pairs $\beta,\lambda$, the image
  $\beta^{r_\lambda}$, taking total time $O(n\log|G|\log n)$.
  \item[(c)] At most $\log|G|$ times, 
    a candidate block will be tested, taking time $O(n|S|)$ per test.
    If it is not a block, then a 
    subgroup $H\leq G_\alpha$, which
    is maintained by the algorithm, must be increased.
    After a candidate
    block is tested and found not to be a block, the algorithm
    finds an element $g\in G_\alpha\setminus H$
    in time $O(n\log^2|G|)$ per test.  The subgroup
    $H$ is then replaced by $\langle H,g\rangle$.
    The total time for testing blocks and finding
    elements of $G_\alpha\setminus H$ is $O(n\log^3|G|+n|S|\log|G|)$.  
 \item[(d)] The total time spent post-processing updates to
   $H$ is $O(n\log n)$ per update, for a total of $O(n\log|G|\log n)$.
\end{itemize}
These are enumerated to correspond with the items
in~\cite{seress-book}[Lemma 5.5.6 (a)--(d)].  We will first argue
that we may replace every occurance of $\log|G|$ with
$\mu$, and be sure that if the algorithm reaches the
time bound without completing, then it will have
produced a certified nonredundant base of size $L+1$.

We have already seen this for (a), in Lemma~\ref{transversal}.
For (b), we apply Proposition~\ref{image}.
Since $|{X_1}^*|\leq 2\mu$,
the time is $O(n\mu\log n)$ as desired.

Item (c) is a little
more complicated since there are multiple factors of $\log|G|$:
there are up to $\log |G|$ updates to the subgroup $H$, each of which
takes time $O(n\log^2|G|+n|S|)$. We must specify exactly
how $H$ is managed: let $H=\langle{X_2}^*\rangle$.
Whenever the algorithm needs to replace $H$ with $\langle H,g\rangle$
for some $g\in G_\alpha$, we simply sift
the element $g$.
By Lemma~\ref{augment-generators}, this has the desired effect
of replacing $H$ by $\langle H,g\rangle$.
This is only done when $g\in G_\alpha\setminus H$ (in fact
when $g$ does not preserve the orbits of $H$), so $|{X_2}^*|$ increases.
This cannot occur more than $L\log n$ times before $\ell$ exceeds $L$,
so the number of updates to $H$ is $\leq \mu$.  This shows
that the $O(n|S|\log|G|)$ term can be replaced by $O(n|S|\mu)$,
and that the  $O(n\log|G|\log n)$ term in item (d) can
be replaced by $O(n\mu\log n)$.

That leaves the $O(n\log^3 |G|)$ term in (c), for
finding an element $g\in G_\alpha\setminus H$.
That is, time $O(n\log^2|G|)$ for each $H$.
This occurs when a candidate block has been found not
to be a block. The testing procedure will have
found, for some subgroup $K$ containing $H$,
points $\beta,\gamma\in {\alpha}^K$,
and an element $g_1$ such that $\beta^{g_1}=\gamma$ but
${\alpha}^{Kg_1}\neq{\alpha}^K$.
Given $K$, $\beta$, $\gamma$, and $g_1$,
the task of computing $g\in G_\alpha\setminus H$ consists of:
\begin{enumerate}
  \item Compute a transversal for  $K:K_\alpha$.
  \item Constructing transversal elements $s$ and $t$
    with ${\alpha}^s=\beta$ and ${\alpha}^t=\gamma$.
  \item Setting $g=sg_1t^{-1}\in G_\alpha$.
\end{enumerate}
The resulting $g$ must lie outside $H$,
because ${\alpha}^{Kg}\neq{\alpha}^K$.

With deep sifting in use,
we see by Lemma~\ref{transversal} that in time $O(n\mu^2)$
we can either compute the transversal for  $K:K_\alpha$ or achieve $\ell>L$.
The transversal consists of words of length $\leq 2\mu$, so
the elements $s$ and $t$ require $O(n\mu)$
time to construct.  Once we have $s$ and $t$,
it is $O(n)$ work to compute $g$.
We have now verified that the $O(n\log^3 |G|)$ term in (c) can indeed be
replaced by $O(n\mu^3)$.

However, if we account more carefully, it turns out that
this term can actually be replaced by $O(n\mu^2\log n)$.
This is because $K$ contains $H$, indeed $K=\langle H,r_\lambda\rangle$
for some $\lambda$ in the candidate block.
When we compute a transversal for $K:K_\alpha$, we
are able to use deep sifting with the current ${X_2}^*$, and with $X_1$
temporarily replaced by the single element list $X_1(K)=(r_\lambda)$.

For any given $K$ there are at most
$\log n$ sifts which increase $|X_1(K)|$. Recall that when computing
a transversal, every sift increases the deep cube, so those sifts which
do not increase $|X_1(K)|$ must increase $|{X_2}^*|$.
During the entire run of the algorithm there can be at most
$\mu\log n$ sifts which increase some $|X_1(K)|$,
and at most $\mu$ sifts which increase $|{X_2}^*|$.
The total
time spent computing transversals for the various $K:K_\alpha$
is thus $O(n\mu^2\log n)$.
We have:

\begin{lemma}\label{truncated}
  Let $L$ be an arbitrary positive integer,
  and let $G\leq\Sym(\Omega)$ be a
  transitive group given by
  a set $S$ of generators.
  Let $\mu=\min(L\log n,\log|G|)$.
  In time $O(n\mu^2\log n+n|S|\mu)$
  we can do one of the following:
  \begin{enumerate}
  \item Establish that $G$ is primitive.
  \item Establish that $G$ is imprimitive, finding a
    block system for $G$.
  \item Compute a nonredundant partial base  $\beta_1,\ldots,\beta_{L+1}$ of size $L+1$,
    together with, for each $i=1,\ldots,L+1$, some $g_i\in G^{[i]}\setminus G^{[i+1]}$.
  \end{enumerate}

\end{lemma}

Note: for $L\log n>\log|G|$ our time estimate becomes
$O(n\log^2|G|\log n+n|S|\log |G|)$, and the primitivity
algorithm runs to completion, proving Theorem~\ref{mild}.
But besides improving the time for small-base groups, we can
now handle additional cases in nearly linear time.
If the group is not primitive, but the action on the blocks is
as a small-base group, then the ``failure'' output of
the nonredundant
partial base $\beta_1,\ldots,\beta_{L+1}$ and the elements $g_i$
will allow us to compute a block.

\section{Beyond the small-base case}\label{actions}
Suppose that $G$ is imprimitive, but the above modified
primitivity test has exceeded the given base size limit $L$.
On termination, it yields a
certified nonredundant partial base $\beta_1,\ldots,\beta_{L+1}$,
with certificate $g_1,\ldots,g_{L+1}$.
By Lemma~\ref{truncated},
the time required for this is $O(nL^2\log^3 n+n|S|L\log n)$.

\begin{lemma}\label{proper block}
  Let $G\leq\Sym(\Omega)$ have
  a certified nonredundant partial base $\beta_1,\ldots,\beta_{L+1}$,
  with certificate $g_1,\ldots,g_{L+1}$.
  Suppose that $G$ has a block system for which the
  action on blocks has maximum nonredundant base size at most $L$.
  Then with respect to this block system,
  for some $i\in 1,\ldots,L+1$, the points $\beta_i$ and
  ${\beta_i}^{g_i}$ lie in the same block.
\end{lemma}
\begin{proof}
  We will focus on the block system for which $L$ is
  an upper bound on the size of a nonredundant base.
  Let us consider how the $g_i$ act on the blocks.
  Since $g_i$ fixes all $\beta_j$ with $j<i$,
  we have that $g_i$ maps the block of $\beta_j$ to itself for $j<i$.

  We wish to show that for some $i$, the element $g_i$ moves the point $\beta_i$ to another point in the
  same block.  Suppose to the contrary that 
  each $g_i$ moves the block of $\beta_i$ to a different block.   Then we have the nonredundancy
  condition for the action on the blocks: the blocks containing
  the $\beta_i$ are a nonredundant partial base of size $L+1$.
  This is a contradiction, as desired.
\end{proof}

So to find a block system, we only need
find, for each $i$, the smallest block
containing $\{\beta_i,{\beta_i}^{g_i}\}$; one of these must
be a proper block by the above Lemma.
The smallest block containing a given subset of $\Omega$
can be found by an algorithm of Atkinson, Hassan, and
Thorne~\cite{atkinson1984group} in time $o(n|S|\log n)$.
We must do this once for each $i$,
for total time $o(n|S|L\log n)$.
This is subsumed by the $O(n|S|L\log n)$ term above.
The total time is thus $O(nL^2\log^3 n+n|S|L\log n)$.

\subsection{Handling small-base primitive actions}
A recent paper of Kelsey and
Roney-Dougal~\cite{Kelsey2022-xi}
establishes that for primitive groups that are not large,
the size of any nonredundant base is at most $5\log n$.  So we may take
$L=5\log n$, and run the modified
Sch\"onert-Seress primitivity test in time
$O(n\log^5 n+n|S|\log^2 n)$.  If this fails to find a block system,
then we
find the smallest block containing each $\{\beta_i,{\beta_i}^{g_i}\}$
using~\cite{atkinson1984group} as above.
By Lemma~\ref{proper block}, this either finds a block system or establishes that
all primitive actions of $G$
on nontrivial partitions of $\Omega$ are large.
The total time is $O(n\log^5 n+n|S|\log^2 n)$.
This completes the proof of Theorem~\ref{main}.

\subsection{Some large cases in sub-quadratic time}\label{conclusion}

Recall that a large primitive action is contained in the product
action of some $\Sym(m)\wr\Sym(d)$, where the $\Sym(m)$ is acting
on $k$-element subsets of an $m$-set, and the socle of the group is $\Alt(m)^d$
for some $m\geq 5$.  Let us call the triple $(m,k,d)$ the {\em parameters}\/
of the action.
This group is isomorphic to a subgroup of $\Sym(md)$, which
in turn~\cite{BabaiLength,cameron1989chains}
has no chain of proper subgroups of length exceeding $3md/2$.
So the largest nonredundant base
size for such a primitive action is at most $3md/2$.
If this primitive group is acting on a partition of $\Omega$
into sets of size $p\geq 1$, then we have $n=p\binom{m}{k}^d$.
If $dk>2$, then $d\geq 3$ or  $k\geq 3$ or $d=k=2$.
In all of these cases we have $3md/2\leq (9/2)n^{1/3}$.

If we set $L=(9/2)n^{1/3}$, our time is
$O(n^{5/3}\log^3 n+n^{4/3}|S|\log n)$
to either determine primitivity or find a certified nonredundant partial
base of size $L+1$. Again, in the latter case we 
find the smallest block containing each $\{\beta_i,{\beta_i}^{g_i}\}$
using~\cite{atkinson1984group}, all in time $O(n^{5/3}\log^3 n+n^{4/3}|S|\log n)$.
By Lemma~\ref{proper block}, if this fails to find a block then
any action on blocks must posess a nonredundant partial base of size $L+1$.
By the above discussion this is impossible for large actions with $dk>2$ or $md\leq 3n^{1/3}$.
We have:

\begin{theorem}\label{five thirds}
  Let $G\leq\Sym(\Omega)$ be a
  transitive group given by
  a set $S$ of generators.  In time $O(n^{5/3}\log^3 n+n^{4/3}|S|\log n)$
  we can do one of the following:
  \begin{enumerate}
  \item Establish that $G$ is primitive.
  \item Establish that $G$ is imprimitive, finding a
    block system for $G$.
  \item Establish that all primitive actions of $G$
    on nontrivial partitions of $\Omega$ are large, with
    parameters $(m,k,d)$ satisfying $kd\leq 2$ and
    $md> 3n^{1/3}$.
  \end{enumerate}
  (In the third case we do not identify a partition on which
  $G$ acts primitively, nor do we determine whether $G$ is primitive.)
\end{theorem}

The time is below Atkinson's $O(n^2)$, and
this handles the large
primitive actions except the cases in which $dk \leq 2$ and
$md> 3n^{1/3}$.  That is, all cases except:
\begin{enumerate}
\item $d=k=1$, the socle is $\Alt(m)$ in its action on $m$ points.
\item $d=1$, $k=2$, the socle is $\Alt(m)$ in its action on
  $\binom{m}{2}$ $2$-sets of an $m$-set.
\item $d=2$, $k=1$, the socle is
  $\Alt(m)\times\Alt(m)$ in its action on $m^2$ ordered pairs.
\end{enumerate}
It seems that improving upon Atkinson's $O(n^2)$ for these
cases will require a new idea.

\bibliographystyle{plain}
\bibliography{primitivity}

\end{document}